\documentclass[12pt,twoside]{article}
\usepackage[english]{babel}
\usepackage{amsmath}
\usepackage{amsfonts}
\usepackage{amssymb}
\usepackage[all]{xy}
\usepackage{amsthm} 

\newtheorem{theorem}{Theorem}[section]
\newtheorem{proposition}{Proposition}[section]
\newtheorem{corollary}{Corollary}[section]
\newtheorem{lemma}{Lemma}[section]

\newtheorem{remark}{Remark}[section]
\newtheorem*{k lemma}{Key Lemma}
\newtheorem*{lemma A}{Lemma A}
\newtheorem*{lemma B}{Lemma B}
\newtheorem*{keyword}{Key Words}
\newcommand{\C}{\mathbb C}
\newcommand{\R}{\mathbb R}
\newcommand{\Z}{\mathbb Z}

\usepackage[a4paper,tmargin=3 cm,bmargin=3 cm,rmargin=2 cm,lmargin=2 cm]{geometry}

\author{
Abdelhamid Boussejra \thanks{e-mail: boussejra.abdelhamid@uit.ac.ma},  
Nadia Ourchane\thanks{e-mail:nadia.ourchane20@gmail.com} \\
\begin{small}
Department of Mathematics,  Faculty of Sciences
\end{small}\\
\begin{small}
University Ibn Tofail, Kenitra, Morocco
\end{small}}

\date{}   
 
\title{\(L^p\)-Poisson integral representations of the generalized Hua operators  on line bundles over $SU(n,n)/S(U(n)\times U(n))$}

\begin{document}

\maketitle

\begin{abstract}
Let \(\tau_\nu\) (\(\nu \in \Z\)) be a character of \(K=S(U(n)\times U(n))\),  and  \(SU(n,n)\times_K\C\)  the associated homogeneous line bundle over \(\mathcal{D}=\{Z\in M(n,\C): I-ZZ^* > 0\}\). Let \(\mathcal{H}_\nu\) be the Hua operator on the sections of \(SU(n,n)\times_K\C\). Identifying sections of \(SU(n,n)\times_K\C\) with functions on  \(\mathcal{D}\) we transfer the  
operator \(\mathcal{H}_\nu\) to an equivalent   matrix-valued operator \(\widetilde{\mathcal{H}}_\nu\) which acts on \(\mathcal{D}\) .\\
Then for a given  \({\C}\)-valued function  \(F\) on \(\mathcal{D}\) satisfying
\(\widetilde{\mathcal{H}}_\nu F=-\frac{1}{4}(\lambda^2+(n-\nu)^2) F.(\begin{smallmatrix}
I&0\\0&-I
\end{smallmatrix})\)
we prove that \(F\) is the Poisson transform by $P_{\lambda,\nu}$ of some $f\in L^p(S)$, when $1<p<\infty$ or $F=P_{\lambda,\nu}\mu$ for some Borel measure $\mu$ on the Shilov boundary $S$, when $p=1$ if and only if 
\[
\sup_{0\leq r < 1}(1-r^2)^{\frac{-n(n-\nu-\Re(i\lambda))}{2}}\left( \int_S |F(rU)|^p {\rm d}U\right) ^{\frac{1}{p}} < \infty,
\]
provided that the complex parameter \(\lambda\) satisfies  \(i\lambda \notin 2\Z^- +n-2\pm \nu\) and \(\Re(i\lambda)>n-1\).\\ 
This generalizes  the result in \cite{B1} which corresponds to $\tau_\nu$ the trivial representation. 
\end{abstract}
\begin{keyword}
Generalized Hua operator, Poisson transform on line bundles, \(H^p\)-Eigensections.
\end{keyword}
\section{Inroduction and main results}
Let \(\mathcal{D}=G/K\) be a bounded symmetric domain of tube type and \(P_\Xi\) a maximal standard parabolic subgroup of \(G\) with Langlands decomposition \\\(P_\Xi=M_\Xi A_\Xi N_\Xi\).   Let \(\tau_{\nu}\) (\(\nu\in\Z\)) be  a character of \(K\), and $E_\nu=G\times_K {\C}$ the associated homogeneous  line bundle. Let \(\sigma_{\lambda,\nu}\) be a character of \(P_\Xi\) parametrized by \(\lambda\in \C, \nu\in \Z\) such that \(\sigma_{{\lambda,\nu}_{\mid_{K\cap M_\Xi}}}=\tau_{\nu_{\mid_{K\cap M_\Xi}}}\), and \( L_{\lambda,\nu}=G\times_{P_\Xi}\C\) the associated homogeneous line bundle over the Shilov boundary \(S=G/P_\Xi\). The space of hyperfunction-valued sections of this line bundle can be identified with 
\[
B(G/P_\Xi, \sigma_{\lambda,\nu})=\{f\in B(G); f(gp)=\sigma_{\lambda,\nu}(p)^{-1}f(g), \forall g\in G, \, p\in P_\Xi\}
\]
 Then we can define a Poisson transform \(P_{\lambda,\nu}\)from \(B(G/P_\Xi, \sigma_{\lambda,\nu})\) to the space \(C^{\infty}(G/K,\tau_\nu)\) by 
\begin{equation}\label{Poisson1}
P_{\lambda,\nu}f(g)=\int_K \tau_{\nu}(k)f(gk){\rm d}k.
\end{equation}
Here \(C^{\infty}(G/K,\tau_\nu)\) denotes the space of smooth functions on \(G\) that  are right \(K\)-covariant of type \(\tau\).\\
Consider the element of \(\mathcal{U}(\mathfrak{g}_c)\otimes \mathfrak{k}_c \) defined by \(\mathcal{H}_\nu\)=\(\sum_{j,k}E_jE_k^\ast\otimes [E_k,E_j^\ast]\). Then \(\mathcal{H}_\nu\) defines a homogeneous differential operator from \(C^{\infty}(G/K,\tau_\nu)\) to \(C^{\infty}(G/K,\tau_\nu\otimes Ad_K\mid\mathfrak{k}_c)\). The operator \(\mathcal{H}_\nu\) will be called here the generalized Hua operator.\\
In \cite{KZ} Koufany and Zhang proved that the Poisson transform  \(P_{\lambda,\nu}\) is an isomorphism from \(B(G/P_\Xi, \sigma_{\lambda,\nu})\) onto an eigenspace \( \mathcal{E}_{\lambda,\nu}(G/K,\tau_\nu)\) of the generalized Hua operator, under certain condition on \(\lambda\) (see also [\cite{OS}, \cite{B}]).\\
This result suggests the problem of  characterizing the Poisson transforms \(P_{\lambda,\nu}(B_0(S))\) where \(B_0(S)\) is some natural subspace of \(B(G/P_\Xi, \sigma_{\lambda,\nu})\) such as the space \(L^p(G/P_\Xi,\sigma_{\lambda,\nu})\) ....\\
In this paper we will consider this problem for \(\mathcal{D}\) the bounded symmetric domain consisting of \(n\times n\) complex matrices \(Z\) such that \(I-ZZ^* > 0\) (positive definite). \\
To describe our results let us fix some notations referring to the next section for more details. As Hermitian space \(\mathcal{D}=SU(n,n)/S(U(n)\times U(n))\) and the Shilov boundary \(S\) is given by \(S=\{Z: ZZ^\ast =I\}= U(n)\).\\
The right \(K\)-covariant functions may be identified with functions defined on \(\mathcal{D}\) (see Section 3). Under this identification the operator \(\mathcal{H}_\nu\) may be viewed as a differential operator acting from \(C^\infty(\mathcal{D})\) to \(C^\infty(\mathcal{D},\mathfrak{k}_c)\). We realize explicitly this new operator which we denote by \(\widetilde{\mathcal{H}}_\nu \). More precisely, if \(\partial\) is the \(n\times n\)-matrix differential operators \(\left( \frac{\partial}{\partial z_{ij}}\right)_{i,j}\), then we show (Proposition 3.1) that 
\begin{equation*}
\widetilde{\mathcal{H}}_\nu=
(\begin{smallmatrix}
(I-ZZ^{\ast})\overline{\partial}(I-Z^{\ast}Z)\partial^{\prime}-\nu(I-ZZ^{\ast})\overline{\partial}Z^{\ast}&0\\
0&-\partial^{\prime}(I-ZZ^{\ast})\overline{\partial}(I-Z^{\ast}Z)+\nu Z^{\ast}\overline{\partial}(I-ZZ^{\ast})
\end{smallmatrix})
\end{equation*} 
with the convention that \(\overline{\partial}\) and \(\partial^{\prime}\) do not differentiate $ (I-Z^{\ast}Z)$ and $(I-ZZ^{\ast})$ respectively.\\  In the above \(\overline{\partial}\) and \(\partial^{\prime}\) denote  respectively, the conjugate and the transpose of \(\partial\).\\
Now, following \cite{OT}, \cite{K} the Poisson transform $P_{\lambda,\nu}$ may be  writing  explicitly as  integral operator from \(B(S)\) the space  of hyperfunctions on $S$ to $\mathcal{C}^\infty(\mathcal{D})$. Namely (see section 3)
\begin{equation*} 
(P_{\lambda,\nu}f)(Z)=\int_{S} P_{\lambda,\nu}(Z,U)f(U)\,{\rm d}U,
\end{equation*}
where 
\[ P_{\lambda,\nu}(Z,U)=\left(\frac{\det(I-ZZ^{\ast})}{\mid \det(I-ZU^{\ast})\mid^2}\right)^{\frac{i\lambda+n-\nu}{2}}(\det(I-ZU^\ast))^{-\nu}.\]
A simple computation shows (Corollary \ref{eigenvalue} ) that the Poisson integrals \(P_{\lambda,\nu}f\) satisfy the following generalized Hua system of second order differential equations
\begin{equation}\label{H}
\widetilde{\mathcal{H}}_\nu F=-\frac{1}{4}(\lambda^2+(n-\nu)^2) F.(\begin{smallmatrix}
I&0\\0&-I
\end{smallmatrix})
\end{equation}
For \(\nu\in {\Z}\) and \(\lambda\in {\C}\) let \(\mathcal{E}_{\lambda,\nu}({\cal D})\) denote the space of all $\C$-valued functions on $\mathcal{D}$ that satisfy (\ref{H}).\\ Since  \(tr((I-ZZ^{\ast})\overline{\partial}(I-Z^{\ast}Z)\partial^{\prime}-\nu(I-ZZ^{\ast})\overline{\partial}Z^{\ast})\) is -up to a constant- the Casimir operator of \(SU(n,n)\) with respect to a weighted action with weight \(\nu\in {\Z}\), it follows that the elements of \(\mathcal{E}_{\lambda,\nu}({\cal D})\) are real analytic functions on \(\mathcal{D}\), see \cite{IO}.\\
Moreover, for \(\nu\in {\Z}\) and \(\lambda\in {\C}\) such that \(i\lambda \notin 2\Z^- +n-2\pm \nu\), the solutions \(F\) of  (\ref{H}) are given by  generalized Poisson transforms of hyperfunctions \(f\) on the Shilov boundary, see \cite{KZ} and \cite{B}.\\
In this paper we will show that the solutions of the generalized Hua system (\ref{H})  for which the hyperfunction \(f\) is given by a Borel  measure or  an \(L^p\)-function (\(1<p<\infty\)) are characterized by an \(H^p\)-condition.\\
As usual we denote by \(L^p(S)\) the Lebesgue space consisting of all measurable (classes) \(\C\)-valued functions \(f\) on \(S\) with
\[
\Vert f\Vert_p =\left(\int_{S} \mid f(U)\mid^p\, {\rm d}U\right) ^{\frac{1}{p}} < \infty.
\] 
Let $\mathcal{M}(S)$ denote the space of all complex Borel measures $\mu$ on $S$ with norm $\|\mu\|=|\mu|(S)$. Here $|\mu|$ denotes the total variation of the measure $\mu$.\\ 
For \(1\leq p<\infty\), define the Hardy-type spaces  \(\mathcal{E}_{\lambda}^{p,\nu}({\cal D})\) by
\begin{equation*}
 \mathcal{E}_{\lambda}^{p,\nu}({\cal D})=\{F\in \mathcal{E}_{\lambda,\nu}({\cal D}): \|F\|_{*,p}<\infty\},
\end{equation*} 
where 
\begin{equation*}
 \|F\|_{*,p}=\sup_{0\leq r<1} (1-r^2)^{\frac{-n(n-\nu-\Re(i\lambda))}{2}}\left( \int_S |F(rU)|^p {\rm d}U\right) ^{\frac{1}{p}}.
\end{equation*}
Then the  main result of this paper may be stated as follows.
\begin{theorem}\label{1<p<oo}
Let \(\lambda \in \C\) satisfies the following conditions
 \[i\lambda \notin 2\Z^- +n-2\pm \nu \, \textit{and}\quad \Re(i\lambda)>n-1.\] 
If \(F\in \mathcal{E}_{\lambda,\nu}({\cal D})\) then 
\begin{itemize}
\item[(i)] For \(1<p<\infty\), \(F(Z)=\int_{S} P_{\lambda,\nu}(Z,U)f(U)\,{\rm d}U\)  for some \(f\) in \(L^p(S)\)  if and only if \(F\in \mathcal{E}_{\lambda}^{p,\nu}({\cal D})\).\\
Moreover, there exists a positive constant  \(\gamma(\lambda)\) such that for every \(f\in L^p(S)\) the following estimates hold:
\begin{align*}
|c_\nu(\lambda)| \|f\|_p \leq \|P_{\lambda,\nu} f\|_{*,p} \leq \gamma(\lambda)\|f\|_p.
\end{align*}
\item[(ii)] For \(p=1\), \(F(Z)=\int_{S} P_{\lambda,\nu}(Z,U)\,d\mu(U)\) for some complex Borel measure \(\mu\) on \(S\) if and only if \(F\in \mathcal{E}_{\lambda}^{1,\nu}({\cal D})\).\\
Moreover, there exists a positive constant  $ \gamma(\lambda) $ such that for every $\mu \in \mathcal{M}(S)$ the following estimates hold:
\begin{align*}
|c_\nu(\lambda)| \|\mu\| \leq \|P_{\lambda,\nu} \mu\|_{*,1} \leq \gamma(\lambda)\|\mu\|.
\end{align*}
\end{itemize}
\end{theorem}
In above, $c_\nu(\lambda)$ is the $c$-function given by 
\[
c_\nu(\lambda)=\frac{\Gamma_\Omega(n)\Gamma_\Omega(i\lambda)}{\Gamma_\Omega(\frac{i\lambda+n+\nu}{2})\Gamma_\Omega(\frac{i\lambda+n-\nu}{2})},
\]
where  \(\Gamma_\Omega(s)=\prod_{j=1}^n \Gamma(s-(j-1))\) is the Gindikin Gamma function (see \cite{FK}).\\ 
\begin{remark}
(i) For \(\nu=0\) and \(p=2\) the above theorem was proved by the first author \cite{B1}.\\
(ii) For \(\nu=0\) the first author and Koufany \cite{B2,B3} generalized Theorem \ref{1<p<oo} to all Hermitian symmetric spaces.
\end{remark}
As an immediate consequence of our main result, we get a Riesz-Herglotz type theorem for the generalized Hua-Harmonic functions on \(\cal D\)(i.e., \(\widetilde{\mathcal{H}}_\nu F=0\)). 
\begin{corollary}
Let \(\nu\in \Z^-\).\\
(i) (Riesz-Herglotz representation type formula) The Poisson transform \(P_{i(\nu-n),\nu}\) is a topological isomorphism from \(\mathcal{M}(S)\) onto \( \mathcal{E}_{i(\nu-n)}^{1,\nu}({\cal D})\).\\
(ii) For \(1<p<\infty\), the Poisson transform \(P_{i(\nu-n),\nu}\) is a topological isomorphism from \(L^p(S)\) onto \(\mathcal{E}_{i(\nu-n)}^{p,\nu}({\cal D})\).
\end{corollary}
We prove  the necessary conditions (Theorem \ref{1<p<oo}) by making use of  the generalized Forelli-Rudin inequality (Proposition \ref{prop CN}). For the sufficiency conditions, the crucial point is the reduction to the case $p=2$, namely 
\begin{theorem}\label{p=2}
Let  \(\nu \in \Z\) and $\lambda \in \C$ such that $i\lambda \notin 2\Z^- +n-2\pm \nu$ and $\Re(i\lambda)>n-1$ and let $F\in \mathcal{E}_{\lambda,\nu}({\cal D})$. Then 
\begin{itemize}
\item[(i)] $F$ is the generalized Poisson transform by $P_{\lambda,\nu}$ of some $f$ in $L^2(S)$  if and only if $F\in \mathcal{E}_{\lambda}^{2,\nu}({\cal D})$.\\
Moreover, there exists a positive constant  $ \gamma(\lambda) $ such that for every $f\in L^2(S)$ the following estimates hold:
\begin{align*}
|c_\nu(\lambda)| \|f\|_2 \leq \|P_{\lambda,\nu} f\|_{*,2} \leq \gamma(\lambda)\|f\|_2.
\end{align*}
\item[(ii)] If $F\in \mathcal{E}_{\lambda}^{2,\nu}({\cal D})$. Then its $L^2$-boundary value $f$ is given by the following inversion formula:
\begin{align*}
f(U)=| c_\nu(\lambda)|^{-2}\lim_{r\rightarrow 1^-}(1-r^2)^{-n(n-\nu-\Re(i\lambda))}\int_S \overline{P_{\lambda,\nu}(rV,U)}F(rV){\rm d}V,
\end{align*} 
in $L^2(S)$.
\end{itemize}
\end{theorem}
In order to prove the "if part" of Theorem \ref{p=2}, we first compute explicitly the series expansion of any eigenfunction $F\in \mathcal{E}_{\lambda,\nu}({\cal D})$ near the $K$-orbits $KA_\Xi.0$ and so we are led to establish a uniform asymptotic behavior of  the following Hua type integrals
\begin{equation*}
\Phi_{\lambda,\textbf{m}}^\nu(r)=\int_S P_{\lambda,\nu}(rI,v)\phi_\textbf{m}(v)\,{\rm d}v,
\end{equation*}
where $\phi_{\textbf{m}}$ are the Schur functions with signature $\textbf{m}=(m_1,...,m_n)$. We show (Proposition \ref{prop gene spherical}) that  \(\Phi_{\lambda,\textbf{m}}^\nu(r)\) may be  given in terms of the classical Gauss hypergeometric functions \(_2F_1(a,b;c;x)\). Namely 
\begin{equation*}
\Phi_{\lambda,\textbf{m}}^\nu(r)=\frac{n!}{d_{\textbf{m}}}\det(\phi^\nu_{\lambda,(m_{i}-i+j)}(r))_{1\leq i,j\leq n},
\end{equation*}
where
\begin{equation*}
\begin{split}
&\phi_{\lambda,k}^\nu(r)=r^{\mid k\mid} (1-r^{2})^{\frac{i\lambda+n-\nu}{2}} \frac{(\frac{i\lambda+n+\epsilon(k)\nu}{2})_{\mid k\mid}}{(1)_{\mid k\mid}}\times \\
_2F_1(&\frac{i\lambda+n-\epsilon(k)\nu}{2},\frac{i\lambda+n+\epsilon(k)\nu}{2}+\mid k\mid;1+\mid k\mid;r^{2}), k\in {\Z}
\end{split}
\end{equation*}
with \(\epsilon(k)=\pm 1\) according to \(k\) is positive or non positive integer.\\ 
To get the  desired asymptotic behavior of the generalized spherical functions \(\Phi_{\lambda,\textbf{m}}^\nu\) (Key Lemma) we prove (see Appendix) the following   results on the determinant of hypergeometric functions.  
\begin{lemma A}\label{tec1}
Let $\alpha,\beta$ be two complex numbers such that $\alpha,\alpha+\beta \notin \{1-k: 1\leq k \leq n-1 \}$. For any $n$-tuple $\textbf{p}=(p_1,p_2,\ldots,p_n)$ of complex numbers 
\begin{equation*}
\begin{split}
\det(_2F_1(\alpha,\beta+p_i+j;&\alpha+\beta;1-r^2))_{i,j}=(-1)^{\frac{n(n-1)}{2}}(1-r^2)^{\frac{n(n-1)}{2}}\prod_{k=1}^{n-1}\left(\frac{\alpha+k-1}{\alpha+\beta+k-1}\right)^{n-k}\times \\ 
&\det(_2F_1(\alpha+n-j,\beta+p_i+n,\alpha+\beta+n-j,1-r^2))_{i,j},
\end{split}
\end{equation*}
\end{lemma A}

\begin{lemma B}\label{tec2}
If $\alpha,\beta$ are two complex numbers such that $\beta \notin \{1-k: 1\leq k \leq n-1\}$ and $\alpha+\beta \notin\{1-k: 1\leq k \leq 2(n-1)\}$. Then  
\begin{equation*}
\begin{split}
\frac{1}{d_\textbf{p}}\det({}_2F_1(\alpha+n-j,\beta+p_i+n;&\alpha+\beta+n-j;1-r^2))_{i,j}\sim (-1)^{\frac{n^2+3n-8}{2}} \prod_{k=1}^{n-1}(n-k)! \times \\
& (1-r^2)^\frac{n(n-1)}{2}\prod_{k=1}^{n-1}\frac{(\beta+k-1)^{n-k}}{\displaystyle \prod_{j=1}^{n-k}(\alpha+\beta+n+k-j-2)_2},
\end{split}
\end{equation*}
as $r$ goes to $1^-$, uniformly in \(\textbf{p}\in \C^n\).
\end{lemma B} 
In above  \(\displaystyle d_\textbf{p}=\prod_{1\leq i<j\leq n}\frac{p_i-p_j}{j-i}\), \(p_i\neq p_j\).\\
Before giving the outline of this paper, We should notice to the reader that our methods in this paper are explicit and self-contained. It is our belief that the main results of this paper may be extended to all bounded symmetric domains of tube type.
We intend to return to this problem in the near future.\\
The organization of this paper is as follows. In the next section we review the bounded realization of \(SU(n,n)/S(U(n)\times U(n))\) and give the explicit matrix realization of the generalized Hua operator. In section 3, we define the Poisson transform on line bundles and give the explicit action of the generalized Poisson transform on \(K\)-types. In section 4 we prove  a uniform asymptotic behavior of the generalized spherical function. The last section is devoted to the proof of our results and we end this paper by an appendix where the proofs of Lemma A and Lemma B are given.
\section{The matrix realization of the generalized Hua operators}
\subsection{Preliminaries}
In this section we review some known results of harmonic analysis on $G=SU(n,n)$.\\
Denote by \(M(N,\C)\) the space of all \(N\,\mbox{by}\,N\) complex matrices. Let $G=SU(n,n)$ be the group of all $g\in M(2n,\C)$ with $\det g=1$ keeping invariant the hermitian form
\begin{equation*}
J(z)=\sum^n_{j=1}\mid z_j\mid^2-\sum^{2n}_{j=n+1}\mid z_j\mid^2.
\end{equation*}
As such form, we take the matrix $J=\begin{pmatrix}
I&0\\0&-I
\end{pmatrix}$, where $I$ is the identity matrix of size $n$.\\
Let
\begin{equation*}
K=S(U(n)\times U(n))=\left\lbrace
\begin{pmatrix}
A&0\\0&D
\end{pmatrix}: A, D\in U(n), \, \det(AD)=1
\right\rbrace ,
\end{equation*}
The group \(G\) acts transitively on \(\cal D\) by 
\begin{equation}\label{action G,D}
g.Z=(AZ+B)(CZ+D)^{-1}, \quad \textit{if} \quad g=\begin{pmatrix}
A&B\\C&D
\end{pmatrix},
\end{equation}
and we have the identification \({\cal D}=G/K\) as a homogeneous space.\\ The Shilov boundary \(S\) of \(\cal{D}\) is the unitary group \(U(n)\). The action (\ref{action G,D}) extends naturally to \(\overline{\cal{D}}\) and under this action \(K\) acts transitively on the Shilov boundary \(S\). As a homogeneous space we have \(S=K/L\) where \(L\) is the stabilizer in \(K\) of \(I\) and consists of the following block matrices \(\begin{pmatrix}A&0\\0&A\end{pmatrix}\) where \(A\in U(n)\) with \(\det A=\pm 1\).\\
Let $\mathfrak{g}=su(n,n)$ be the Lie algebra of $G$, under the Cartan involution $\theta(X)=JXJ$ we have $\mathfrak{g}=\mathfrak{k}+\mathfrak{p}$ where  
\begin{equation*}
\mathfrak{p}=\left\lbrace \begin{pmatrix}
0&Z\\
Z^\ast&0
\end{pmatrix}: Z\in M(n,\C)\right\rbrace,
\end{equation*}
and 
\begin{equation*}
\mathfrak{k}=\left\lbrace \begin{pmatrix}
X&0\\0&Y 
\end{pmatrix}: X^\ast+X=Y^\ast+Y=0, \, tr(X+Y)=0\right\rbrace,
\end{equation*}
$\ast$ denoting  conjugate transpose  and $tr$ stands for the trace.\\
For real Lie algebra \(b\) we denote by \(b_c\) its complexification. We have \(\mathfrak{g}_c= \mathfrak{p}_c\oplus \mathfrak{k}_c\)  where \(\mathfrak{g}_c=sl(2n,\C)\),
\begin{equation*}
\mathfrak{p}_c=\left\lbrace \begin{pmatrix}
0&Z\\
W&0
\end{pmatrix}: Z, W\in M(n,\C)\right\rbrace,
\end{equation*}
and
\begin{equation*}
\mathfrak{k}_c=\left\lbrace \begin{pmatrix}
X&0\\
0&Y
\end{pmatrix}: X,Y\in M(n,\C), \, tr(X+Y)=0\right\rbrace .
\end{equation*}
As a Cartan subalgebra of \(\mathfrak{k}\) we choose \(\mathfrak{t}\) to be the set of diagonal matrices. Let \(e_j\) be the linear form on \(\mathfrak{t}_c\) which select the jth entry of the diagonal matrix. We choose an ordering of the roots \(\Delta(\mathfrak{g}_c,\mathfrak{t}_c)\) in which \(\mathfrak{p}^+\) has the form 
\begin{equation*}
\mathfrak{p}^+=\left\lbrace \begin{pmatrix}
0&Z\\
0&0
\end{pmatrix}: Z\in M(n,\C)\right\rbrace ,
\end{equation*} 
and 
\begin{equation*}
\mathfrak{p}^-=\left\lbrace \begin{pmatrix}
0&0\\
Z&0
\end{pmatrix}: Z\in M(n,\C)\right\rbrace ,
\end{equation*}
Thus \(\Delta^+(\mathfrak{g}_c,\mathfrak{t}_c)=\{e_i-e_j: 1\leq i<j\leq 2n\}\).\\
Let \(\Delta^+_n\) be the set of positive noncompact roots. Then \(\{e_j-e_{n+j}: 1\leq j\leq n\}\) is the maximal set of strongly orthogonal noncompact positive roots. Let \(E_{ij}\) denote the \(2n\times 2n\) matrix with entry \(1\) at the \((i,j)\) entry, all other entries being \(0\). Then \(\mathfrak{a}=\displaystyle\sum_{j=1}^n{\R}(E_{j,n+j}+E_{n+j,j})\) is a maximal abelian subalgebra in \(\mathfrak{p}\). 
Thus, we have \(\mathfrak{a}=\left\lbrace H_T=\begin{pmatrix}
0&T\\
T&0
\end{pmatrix}: T=diag(t_1,\cdots,t_n),\, t_j\in {\R}\right\rbrace \).\\ 
Define \( \alpha_j\in \mathfrak{a}^\ast\) by  $\alpha_j(H_T)=t_j$. Let  $\Sigma$ denote the restricted roots system  of $(\mathfrak{g},\mathfrak{a})$ and \(\Sigma^+\) the positive system of \(\Sigma\). Then  
\(
\Sigma=\{\pm 2\alpha_j (1\leq j\leq n), \pm\alpha_i\pm\alpha_j (1\leq i\neq j\leq n)\}\),
and 
\(\Sigma^+=\{2\alpha_j (1\leq j\leq n), \alpha_j\pm\alpha_i (1\leq i<j\leq n\}\), with multiplicities $m_{2\alpha_j}=1$ and $m_{\alpha_i\pm\alpha_j}=2$. 
The set \(\Gamma=\{2\alpha_1,(\alpha_2-\alpha_1),\cdots,(\alpha_n-\alpha_{n-1})\}$ is the set of simple roots in \(\Sigma^+\).\\
We set $\Xi=\Gamma\setminus\{2\alpha_1\}$ and let $\mathfrak{a}_\Xi=\{H\in\mathfrak{a}: \Xi(H)=0\}$. Then $\mathfrak{a}_\Xi={\R}X_0$, where $X_0=\begin{pmatrix}
0&I\\
I&0
\end{pmatrix}$. Let $\rho_{0}$ and $\rho_{\Xi}$ be the linear forms on $\mathfrak{a}_{\Xi}$  defined by $\rho_{0}(X_{0})=n$ and $\rho_{\Xi}=n\rho_{0}$.\\
Let $P_\Xi$ be the standard parabolic subgroup of $G$ associated to $\Xi$. 
Let $P_\Xi=M_\Xi A_\Xi N_\Xi$ be the Langlands decomposition of \(P_\Xi\) such that $A_\Xi\subset A$. Then 
\begin{equation*}
M_\Xi=\left\lbrace \begin{pmatrix}
m_1&m_2\\
m_2&m_1 
\end{pmatrix}\in G: \, \mid \det (m_1+m_2)\mid=1\right\rbrace ,
\end{equation*}
\begin{equation*}
A_\Xi=\left\lbrace \begin{pmatrix}
\cosh tI&\sinh tI\\
\sinh tI&\cosh tI
\end{pmatrix}: t\in{\R}\right\rbrace ,
\end{equation*}
and 
\begin{equation*}
N_\Xi=\left\lbrace \begin{pmatrix}
I+X&-X\\
X&I-X
\end{pmatrix}: X\in M(n,{\C}),\, X+X^\ast=0\right\rbrace .
\end{equation*}
We have \(K\cap M_\Xi=L\) and  \(G/P_\Xi\simeq K/K\cap M_\Xi\).\\
Finally we have the generalized Iwasawa decomposition $G=KM_{\Xi}A_{\Xi}N_{\Xi}$. Each $g\in G$ can be written as 
\begin{equation*}
g=\kappa(g)\mu(g)\exp(H_\Xi(g))n,
\end{equation*}
where $\kappa(g)\in K$, $\mu(g)\in M_{\Xi}$, $H_\Xi(g)\in \mathfrak{a}_\Xi$ and $n\in N_{\Xi}$, with uniqueness  of $H_\Xi(g), n$ and  $\kappa(g)\mu(g)$.\\
\subsection{The matrix realization of the generalized Hua operator}
In this section we give the matrix realization of the generalized Hua operator $\widetilde{\mathcal{H}}_\nu$.\\
Let $P^+, P^-$ and $K_c$ be the analytic subgroups of $G_c=SL(2n,{\C})$ with Lie algebras $\mathfrak{p}^+,\mathfrak{p}^-$ and $\mathfrak{k}_c$ respectively. Then \(P^+K_cP^-\) is a dense open subset of \(SL(2n,{\C})\) containing \(G\). Furthermore  if \(g\in G\) is writing in block form  \(g=\begin{pmatrix} 
A&B\\
C&D
\end{pmatrix}
\) then 
\(g\) may be written uniquely with respect to the  Harish-Chandra decomposition \(P^+K_cP^- \) as 
\begin{equation}
\begin{pmatrix}
A&B\\
C&D
\end{pmatrix}=
\begin{pmatrix} I&BD^{-1}\\
0&I
\end{pmatrix}\begin{pmatrix}
A-BD^{-1}C&0\\
0&D
\end{pmatrix}\begin{pmatrix}
I&0\\
D^{-1}C&I
\end{pmatrix}
\end{equation}
In the sequel we  denote by \(\pi_0(g)\) the \(K_c\)-component of \(g\). Namely \(\pi_0(g)=\begin{pmatrix}
A-BD^{-1}C&0\\
0&D
\end{pmatrix}\).\\
The $G$-homogeneous line bundles $E_\nu=G\times_K{\C}$ are given by the characters of $K$. Any  character of $K$ is  determined by $\nu \in{\Z}$ and is given by 
\begin{equation*}
\tau_\nu (\begin{pmatrix}
A&0\\
0&D
\end{pmatrix})=(\det D)^\nu.
\end{equation*}

If \((\rho,V)\) is a finite dimensional representation of \(K\), sections of the homogeneous vector bundle \(G\times_K V\) can be regarded as elements of   \(C^\infty(G,\rho)\) the space of smooth right \(K\)-covariant functions \(f\) on \(G\), i.e. \(f(gk)= \rho(k)^{-1}f(g), (g\in G, k\in K)\). For $f\in C^{\infty}(G,\rho)$  define the function $\theta_\rho f:\mathcal{D}\rightarrow {\C}$  by  
\begin{equation*}
(\theta_\rho f)(Z)=\rho(\pi_0(g))f(g), 
\end{equation*}
then via \(\theta_\rho\) right \(K\)-covariant functions are  naturally  identified with functions on \(\mathcal{D}\) and the opertor \(\mathcal{H}_\nu\) may be viewed as acting from \(C^\infty(\mathcal{D})\) to \(C^\infty(\mathcal{D},\mathfrak{k}_c)\). In this subsection we give explicitly this operator which we denote by \(\widetilde{\mathcal{H}}_\nu\).\\
Write \(X\in \mathfrak{g}_c\) as \(X=X_0+iY_0\) with \(X_0,Y_0\in \mathfrak{g}\). If \(f\) is a differentiable function on \(G\), define the function $Xf$ by 
\begin{equation*}
Xf(g)=\frac{d}{dt}f(g{\rm e}^{tX_0})_{\mid t=0}+i\frac{d}{dt}f(g{\rm e}^{tY_0})_{\mid t=0},
\end{equation*}
For $U$ in $M(n,\C)$ define $L^{+}(U)=\begin{pmatrix}
0&U\\0&0
\end{pmatrix}$ and $L^{-}(U)=\begin{pmatrix}
0&0\\U^*&0
\end{pmatrix}$.\\
Then $L^{\pm}$ are ${\C}$-linear isomorphism from $M(n,{\C})$ to $\mathfrak{p}^{\pm}$.\\
We denote by $E_{ij}$ the $n\,\mbox{by}\,n$ matrix with $1$ in the $(i,j)$ position and $0$ everywhere. We normalize the Killing form so that $L^+(E_{ij})$ is an orthonormal  basis of $\mathfrak{p}^+$ and $L^-(E_{ij})$ is its dual basis in $\mathfrak{p}^-$.\\
Then the Hua operator $\mathcal{H}_\nu$  is the homogeneous differential operator from $C^{\infty}(E_{\nu})$ to $C^{\infty}(G\times_K \C\otimes \mathfrak{k}_c)$  defined by 
\begin{equation*}\label{Hua}
 \mathcal{H}_\nu=\sum_{\begin{array}{c} 
\scriptstyle 1\leq i,j \leq n \\ 
\scriptstyle 1\leq p,q \leq n 
\end{array}}L^+(E_{ij})L^-(E_{pq})\otimes [L^+(E_{pq}),L^-(E_{ij})].
\end{equation*}
As explained before, we transfer the operator $ \mathcal{H}_\nu$  to an equivalent   matrix-valued operator 
\(\widetilde{\mathcal{H}}_\nu\) which acts on \(\mathcal{D}\). Namely, the operator \(\widetilde{\mathcal{H}}_\nu\) acting from \(C^\infty(\mathcal{D})\) to \(C^\infty(\mathcal{D}, \mathfrak{k}_c)\) is given by 
\begin{equation*}
\widetilde{\mathcal{H}}_\nu=\theta_{\tau_{\nu}\otimes Ad} \circ \mathcal{H}_\nu\circ \theta_{\tau_\nu}^{-1}.
\end{equation*}
Finally, we define the differential operator \(\partial: C^\infty(\mathcal{D})\rightarrow C^\infty(\mathcal{D}, M(n,\C))\)    as follows
\begin{eqnarray*}
\partial F=\left(\frac{\partial F}{\partial z_{ij}}\right)_{1\leq i,j\leq n}
\end{eqnarray*}
Then the  main result of this section is as follows
\begin{proposition}\label{Hua operator}
The generalized Hua operator $\widetilde{\mathcal{H}}_\nu$ is given by
\begin{equation*}
\widetilde{\mathcal{H}}_\nu=
\left(\begin{smallmatrix}
(I-ZZ^{\ast})\overline{\partial}(I-Z^{\ast}Z)\partial^{\prime}-\nu(I-ZZ^{\ast})\overline{\partial}Z^{\ast}&0\\
0&-\partial^{\prime}(I-ZZ^{\ast})\overline{\partial}(I-Z^{\ast}Z)+\nu Z^{\ast}\overline{\partial}(I-ZZ^{\ast})
\end{smallmatrix}\right)
\end{equation*}
(with the understanding that \(\overline{\partial}\) and \(\partial^{\prime}\) do not differentiate the matrices \(I-Z^{\ast}Z\) and \(I-ZZ^{\ast}\)).
\end{proposition}

We shall need the following lemma.

\begin{lemma}\label{lemm hua op}
For every $U\in M(n,\C)$ and $g=\begin{pmatrix}
A&B\\
C&D
\end{pmatrix} \in G$, we have
\begin{equation}\label{E1}
 L^+(U)\tau_\nu^{-1}(\pi_0(g))=-\nu(\det D)^{-\nu}tr(D^{-1}CU)
\end{equation}
 and
 \begin{equation}\label{E2}
 L^-(\overline{U})\tau_\nu^{-1}(\pi_0(g))=0.
 \end{equation}
\end{lemma}

\begin{proof}
Let $U\in M(n,\C)$. Put $L^+(U)=\frac{1}{2}(X_0-iY_0)$ and $L^{-}(\overline{U})=\frac{1}{2}(X_0+iY_0)$, then
\begin{equation*}
X_0=\begin{pmatrix}
0&U\\
U^*&0
\end{pmatrix}\quad \mbox{and} \quad Y_0=\begin{pmatrix}
0&iU\\
-iU^*&0
\end{pmatrix}.
\end{equation*}
Fix $g=\begin{pmatrix}
A&B\\
C&D
\end{pmatrix} \in G$. If $X=\begin{pmatrix}
\alpha&\beta\\
\gamma&\delta
\end{pmatrix} \in \mathfrak{g}$ then 
\begin{equation}\label{*}
X\tau_\nu^{-1}(\pi_0(g))=-\nu \tau_\nu^{-1}(\pi_0(g))tr(D^{-1}(C\beta+D\delta)).
\end{equation}
Applying (\ref{*}) to $X_0$ and $Y_0$, we easily get 
\[X_0\tau_\nu^{-1}(\pi_0(g))=-\nu \tau_\nu^{-1}(\pi_0(g))tr(D^{-1}CU),\] 
and 
\[Y_0\tau_\nu^{-1}(\pi_0(g))=-i\nu \tau_\nu^{-1}(\pi_0(g))tr(D^{-1}CU),\]
and the result follows.
\end{proof}

We also need the following result. 

\begin{lemma}\label{L+-}
Let \(F\) be a differentiable function on \(\mathcal{D}\). Then for every $U\in M(n,{\C})$,
\begin{itemize}
\item[i)] \((L^+(U)\theta_0^{-1} F)(g)=tr(\partial^{\prime} F(g.0)(A^\ast)^{-1}U D^{-1})\)
\item[ii)]\((L^-(\overline{U})\theta_0^{-1} F)(g)=tr(\overline{\partial}F(g.0)\overline{(A^\ast)^{-1}U D^{-1}})\)
\end{itemize}
with $g=\begin{pmatrix}
A&B\\
C&D
\end{pmatrix} \in G$.
\end{lemma}
\begin{proof} 
Let \(X=\begin{pmatrix}
\alpha&\beta\\
\gamma&\delta
\end{pmatrix} \in \mathfrak{g}\). Put \(Z_X(t)=\exp tX.0\). By differentiation of \(g.Z_X(t)\) we obtain
\[ d (g.Z_X(t))_{\mid t=0}=(A^\ast)^{-1}\beta D^{-1}\, {\rm d}t,\]
where $g=\begin{pmatrix}
A&B\\
C&D
\end{pmatrix} \in G$, from which we deduce 
\begin{equation}\label{E}
X\theta_0^{-1}F(g)=tr\left(\partial^{\prime} F(g.0)(A^\ast)^{-1}\beta D^{-1}+\overline{\partial}F(g.0)\overline{(A^\ast)^{-1}U D^{-1}}\right).
\end{equation}
Next applying (\ref{E}) to \((L^{+}(U)+L^{-}(\overline{U}))\) and \(i(L^+(U)-L^{-}(\overline{U}))\) we get the desired result and the proof is finished.
\end{proof}
Now we come back  to the proof of Proposition \ref{Hua operator}.\\
\begin{proof} 
For $F\in \mathcal{C}^\infty(\mathcal{D})$ by definition
\[
(\widetilde{\mathcal{H}}_\nu F)(g.0)=\tau_\nu(\pi_0(g)) \sum_{\begin{array}{c} 
\scriptstyle 1\leq i,j \leq n \\ 
\scriptstyle 1\leq p,q \leq n 
\end{array}}\left( Ad(\pi_0(g)^{-1})L^{+}(E_{ij}) Ad(\pi_0(g)^{-1})L^{-}(E_{pq}) \theta_{\tau_\nu}^{-1}F\right)(g)\otimes [L^{+}(E_{pq}),L^{-}(E_{ij})].
\]
We have 
\begin{equation}\label{**}
\left( Ad(\pi_0(g)^{-1})L^{-}(E_{pq})\theta_{\tau_\nu}^{-1}F\right)(g)= \tau_\nu(\pi_0(g)^{-1})\left( Ad(\pi_0(g)^{-1})L^{-}(E_{pq})\theta_0^{-1}F\right)(g).
\end{equation}
by (\ref{E2}).\\
Then applying $Ad(\pi_0(g)^{-1})L^{+}(E_{ij})$ to (\ref{**}) as a function of $g$ we get 
\begin{eqnarray*}
&&\left( Ad(\pi_0(g)^{-1})L^{+}(E_{ij}) Ad(\pi_0(g)^{-1})L^{-}(E_{pq}) \theta_{\tau_\nu}^{-1}F\right)(g)\\
&=& \left( Ad(\pi_0(g)^{-1})L^{+}(E_{ij}) \tau_\nu(\pi_0(g)^{-1})\right) \left( Ad(\pi_0(g)^{-1})L^{-}(E_{pq}) \theta_0^{-1}F\right)(g)\\
&+&\tau_\nu(\pi_0(g)^{-1}) \left( Ad(\pi_0(g)^{-1})L^{+}(E_{ij})Ad(\pi_0(g)^{-1})L^{-}(E_{pq}) \theta_0^{-1}F \right)(g)\\
&=& I_{ij}^{pq}(g)+J_{ij}^{pq}(g).
\end{eqnarray*}
Next, if \(g\) is writing as block matrices \(g=\begin{pmatrix}
A&B\\
C&D
\end{pmatrix}\) 
a direct calculation gives the following formula
\begin{equation}\label{E3}
Ad(\pi_0(g)^{-1})L^{+}(E_{ij})=L^{+}((A-BD^{-1}C)^{-1}E_{ij}D)
\end{equation}
and
\begin{equation}\label{E4}
Ad(\pi_0(g)^{-1})L^{-}(E_{pq})=L^{-}(^t(D^{-1}E_{pq}(A-BD^{-1}C)))
\end{equation}
Now given $Z\in \mathcal{D}$, define $g_Z$ by
\(
g_Z=\begin{pmatrix}
(I-ZZ^*)^{\frac{-1}{2}}&Z(I-Z^*Z)^{\frac{-1}{2}}\\
Z^*(I-ZZ^*)^{\frac{-1}{2}}&(I-Z^*Z)^{\frac{-1}{2}}
\end{pmatrix}.
\)
Then \(g_Z\in G\) and \(g_Z.0=Z\).\\  
From (\ref{E3}), (\ref{E1}) and  (ii) of Lemma \ref{L+-} it follows that 
\[
I_{ij}^{pq}(g_Z)=-\nu \tau_\nu(\pi_0(g)^{-1}) tr(Z^*(I-ZZ^*)^{-1}E_{ij})((I-ZZ^*)\overline{\partial}F(Z)(I-Z^*Z))_{pq}.
\]

Using (\ref{E4}) and (ii) of  Lemma \ref{L+-}, we easily see that 
\[
Ad(\pi_0(g)^{-1})L^{-}(E_{pq}) \theta_0^{-1}F(g_Z)=((I-ZZ^*)\overline{\partial}F(Z)(I-Z^*Z))_{pq}.
\]

By applying $Ad(\pi_0(g)^{-1})L^{+}(E_{ij})$ to the above equality as a function of $g$, we obtain from (ii) of Lemma \ref{lemm hua op}, 
\[
J_{ij}^{pq}(g_Z)=\tau_\nu(\pi_0(g)^{-1})((I-ZZ^*)\overline{\partial}(I-Z^*Z))_{pq}\frac{\partial F}{\partial z_{ij}}(Z).
\]
Hence
\begin{eqnarray*}\begin{split}
(\widetilde{\mathcal{H}}_\nu F)(Z)&= \sum_{i,j}\sum_{p,q}(-\nu (Z^*(I-ZZ^*)^{-1})_{ji}((I-ZZ^*)\overline{\partial}F(Z)(I-Z^*Z))_{pq}[L^{+}(E_{pq}),L^{-}(E_{ij})]\\
&+\sum_{i,j}\sum_{p,q}((I-ZZ^*)\overline{\partial}(I-Z^*Z))_{pq}\frac{\partial F}{\partial z_{ij}}(Z) [L^{+}(E_{pq}),L^{-}(E_{ij})],
\end{split}\end{eqnarray*}
and the Proposition follows.
\end{proof}

\begin{corollary}\label{eigenvalue}
For $U$ fixed in $S$, the function $Z \rightarrow P_{\lambda,\nu}(Z,U)$ satisfies the generalized Hua system:
\[
\widetilde{\mathcal{H}}_\nu P_{\lambda,\nu}(Z,U)=-\frac{1}{4}(\lambda^2+(n-\nu)^2)P_{\lambda,\nu}(Z,U).J.
\]
\end{corollary}

\begin{proof}
It suffices to establish the following formula \[(I-ZZ^{\ast})\overline{\partial}(I-Z^{\ast}Z)\partial^{\prime}-\nu(I-ZZ^{\ast})\overline{\partial}Z^{\ast}P_{\lambda,\nu}(Z,U)=-\frac{1}{4}(\lambda^2+(n-\nu)^2)P_{\lambda,\nu}(Z,U)I,
\]

From the identity 
\[
\partial(\det(I-ZU^*))^{\beta}=-\beta(\det(I-ZU^*))^{\beta}\,\, {}^t(U^*(I-ZU^*)^{-1}),\quad \beta\in\C,
\]
we easily obtain the following formula 
\[
\partial^{\prime} P_{\lambda,\nu}(Z,U)= P_{\lambda,\nu}(Z,U)\left((\frac{i\lambda+
 n+\nu}{2})U^*(I-ZU^*)^{-1}-(\frac{i\lambda+
 n-\nu}{2})Z^*(I-ZZ^*)^{-1}\right),
\]
and 
\[\overline{\partial}Z^\ast P_{\lambda,\nu}(Z,U)=\frac{i\lambda+
 n-\nu}{2}(I-UZ^\ast)^{-1}U-(I-ZZ^*)^{-1}Z)Z^\ast.
\]
Thus
\[(I-ZZ^{\ast})\overline{\partial}(I-Z^{\ast}Z)\partial^{\prime}-\nu(I-ZZ^{\ast})\overline{\partial}Z^{\ast}P_{\lambda,\nu}(Z,U)=-\frac{1}{4}(\lambda^2+(n-\nu)^2)P_{\lambda,\nu}(Z,U)I,
\]
as to be shown.
\end{proof}
\section{Poisson transform on homogeneous line bundles.}
In this section we consider the Poisson transform on homogeneous line bundles and give the explicit form of the generalized Poisson transform on each \(K\)-types of \(L^2(S)\).
\subsection{The Poisson transform on line bundles}
Consider the character $\sigma_{\lambda,\nu}$ of  $P_\Xi=M_\Xi A_\Xi N_\Xi$ defined by
\[
\sigma_{\lambda,\nu}(man)=\xi_\nu(m)a^{\rho_{\Xi}-i\lambda\rho_{0}},\quad m\in M_\Xi, a\in A_{\Xi}, n\in N_{\Xi},
\]
where  
\[
\xi_\nu(m)=(sign\det(m_1+m_2))^\nu.
\]
Let $L_{\lambda,\nu}=G\times_{P_\Xi}\C$ be the homogeneous line bundle associated to $\sigma_{\lambda,\nu}$. The space \(B(G/P_\Xi; L_{\lambda,\nu})\) may be identified to the space of all hyperfunctions \(f\) on \(G\) that satisfy
\[ f(gman)={\rm e}^{(i\lambda\rho_0-\rho_\Xi)H_\Xi(a)}\xi_\nu(m)^{-1}f(g), \;\; \forall g\in G, m\in M_\Xi, a\in A_\Xi, n\in N_\Xi.\]
Since $G=KP_{\Xi}$, the restriction from \(G\) to \(K\) gives  an isomorphism from \(B(G/P_\Xi; L_{\lambda,\nu})\) onto the space 
\(B(K/K\cap M_\Xi,\tau_\nu)\) of all hyperfunctions \(f\) on \(K\) that satisfy
\[f(km)=\tau_\nu(m)^{-1}f(k), \;\; \forall k\in K, m\in K\cap M_\Xi.\]
 Then a straightforward computation shows that  the Poisson transform of \(h\) in \( B(K/K\cap M_\Xi,\tau_\nu)\) is given by  
\begin{equation*}
P_{\lambda,\nu}h(g)=\int_{K}a_{\Xi}(g^{-1}k)^{-i\lambda-\rho_\Xi}\tau_\nu(\kappa(g^{-1}k))\xi_\nu(\mu(g^{-1}k))h(k)\,{\rm d}k.
\end{equation*}
Since \(B(K/K\cap M_\Xi,\tau_\nu)\simeq B(S)\), it follows that the generalized Poisson transform may be written (we continue to denote it by $P_{\lambda,\nu}$) for \(f\in B(S)\)  as follows
\[P_{\lambda,\nu}f(g)=\int_S a_{\Xi}(g^{-1}k)^{-i\lambda-\rho_\Xi}\tau_\nu(\kappa(g^{-1}k))\xi_\nu(\mu(g^{-1}k))\tau_\nu(k)^{-1}f(k)\, {\rm d}k\]
Next,  let $J_g(0)$ denote the complex Jacobian determinant of the holomorphic map $Z \rightarrow g.Z$, see (\ref{action G,D}). Then  identifying   right \(K\)-covariant functions \(\widehat{F}\) on \(G\) to  functions $F:{\cal D}\rightarrow {\C}$ via  
\(
\widehat{F}(g)=J_g(0)^{-\frac{\nu}{2n}}F(g.0),
\)  
we easily see that the generalized Poisson transform of \(f\in B(S)\) is given by  
\begin{equation*}
P_{\lambda,\nu}f(Z)=\int_S P_{\lambda,\nu}(Z,U)f(U){\rm d}U,
\end{equation*}
with
\begin{equation*}
P_{\lambda,\nu}(Z,U)=J_g(0)^{-\frac{\nu}{2n}} a_{\Xi}(g^{-1}k)^{-i\lambda-\rho_\Xi}\tau_\nu(\kappa(g^{-1}k))\xi_\nu(\mu(g^{-1}k))\tau_\nu(k)^{-1},
\end{equation*}
where \(Z=g.0\) and \(U=k.I\). \\
Now, we will compute the components $\kappa(g)$, $\mu(g)$ and $a_{\Xi}(g)$ for $g=\begin{pmatrix}
A&B\\C&D
\end{pmatrix}$.\\
If
\[\begin{pmatrix}
A&B\\C&D
\end{pmatrix}=\begin{pmatrix}
V&0\\0&W
\end{pmatrix}
\begin{pmatrix}
m_1&m_2\\m_2&m_1
\end{pmatrix}
\begin{pmatrix}
\cosh tI&\sinh tI\\ \sinh tI&\cosh tI
\end{pmatrix}\begin{pmatrix}
I+X&-X\\ X&I-X
\end{pmatrix},\]
is the generalized Iwasawa decomposition of $g$, then we  easily get
\[
\left\{\begin{array}{rcl} A+B=e^t V(m_1+m_2)\\
C+D=e^t W(m_1+m_2)\end{array}.\right. 
\]
This implies that 
\[
a_{\Xi}(g)=\mid \det(C+D)\mid^{\frac{1}{n}},
\] 
and 
\[
\tau_\nu(\kappa(g))\xi_\nu(\mu(g))=\left(\frac{\det(C+D)}{\mid \det(C+D)\mid}\right)^\nu.
\]
Let $k=\begin{pmatrix}
k_1&0\\0&k_2
\end{pmatrix}$
then the two above identities imply
\begin{equation*}
\begin{split}
a_{\Xi}(g^{-1}k)^{-i\lambda-\rho_\Xi}&\tau_\nu(\kappa(g^{-1}k))\xi_\nu(\mu(g^{-1}k))=\left(\frac{\det D}{\mid \det D\mid}\right)^\nu \mid \det D\mid ^{-(i \lambda+n)}\times \\
&(\det k_2)^\nu \left(\frac{\det(I+D^{-1}Ck_1 k_2^{-1})}{\mid \det(I+D^{-1}Ck_1 k_2^{-1})\mid}\right)^\nu.
\end{split}
\end{equation*}
Now because $J_g(0)=(\det D)^{-2n}$, $Z=D^{-1}C$ and $U=k_1 k_2^{-1}$ we deduce that 
\begin{equation*}
P_{\lambda,\nu}(Z,U)=\left( \frac{\det(I-ZZ^\ast)}{\mid \det(I-ZU^\ast)\mid^{2}}\right) ^{\frac{{i\lambda+n-\nu}}{2}}(\det(I-ZU^\ast))^{-\nu}.
\end{equation*}

Thus
\begin{equation}\label{Poisson}
(P_{\lambda,\nu}f)(Z)=\int_S \left( \frac{\det(I-ZZ^\ast)}{\mid \det(I-ZU^\ast)\mid^{2}}\right) ^{\frac{{i\lambda+n-\nu}}{2}}(\det(I-ZU^\ast))^{-\nu}f(U)\,{\rm d}U.
\end{equation}
\begin{remark}
In  \cite{OT} Okamoto et al. computed explicitly the Poisson kernels for Poisson  transforms for homogneous line bundles on Cartan domains by using a different method. Their results have been extended   by Kor\'anyi \cite{K} to all bounded symmetric domains. 
\end{remark}
\subsection{The expansion of the Poisson transform}
Recall that the group \(K\) acts on the space \(L^{2}(S)\) by composition, and under this action the Peter-Weyl decomposition of $L^{2}(S)$ is given by 
\begin{equation*}
L^{2}(S)=\bigoplus_{\textbf{m}\in \Lambda}V_{\textbf{m}},
\end{equation*}
where $\Lambda$ is the set of all $n$-tuple, \(\textbf{m}=(m_{1},m_{2},\cdots,m_{n})\) of integers with $m_{1}\geq m_{2}\geq\cdots\geq m_{n}$ and the $K$-irreducible component $V_{\textbf{m}}$ is the finite linear span $\{\phi_{\textbf{m}}\circ k,k\in K\}$.
the function $\phi_{\textbf{m}}$ is the zonal spherical function associated to the symmetric pair $(K,L)$. More precisely, 
if $\delta=diag(\delta_{1},\delta_{2},\cdots,\delta_{n})$ is a diagonal matrix, then 
\begin{equation*}
\phi_{\textbf{m}}(\delta)=\frac{1}{d_{\textbf{m}}}\frac{A(\textbf{m}+\rho)(\delta)}{A(\rho)(\delta)},
\end{equation*}
with $\rho=(\frac{-(n-1)}{2},\cdots,\frac{(n-1)}{2})$, $A(\mu)(\delta)=\det(\delta_{i}^{\mu_{j}})_{i,j}$ and 
\[
d_{\textbf{m}}=\prod\limits_{1\leq i<j\leq n}\left(1+\frac{m_{i}-m_{j}}{j-i}\right).
\]
For $k\in {\Z}$, let $\phi_{\lambda,k}^\nu$ denote the ${\C}$-valued function on $[0,1[$ given by
\begin{equation*}
\begin{split}
&\phi_{\lambda,k}^\nu(r)=r^{\mid k\mid} (1-r^{2})^{\frac{i\lambda+n-\nu}{2}} \frac{(\frac{i\lambda+n+\epsilon(k)\nu}{2})_{\mid k\mid}}{(1)_{\mid k\mid}} \times\\
 _2F_1(&\frac{i\lambda+n-\epsilon(k)\nu}{2},\frac{i\lambda+n+\epsilon(k)\nu}{2}+\mid k\mid;1+\mid k\mid;r^{2}),
\end{split}
\end{equation*} 
where \((a)_k=a(a+1)\cdots (a+k-1)\) is the Pochammer symbol.
\begin{proposition}\label{prop gene spherical}
If $f\in V_{\textbf{m}}$, then
\begin{equation}\label{K-compo}
P_{\lambda,\nu}f(rU)=\Phi^{\nu}_{\lambda,\textbf{m}}(r)f(U),
\end{equation}
where the generalized spherical function $\Phi_{\lambda,\textbf{m}}^{\nu}$ is given by
\begin{equation}\label{generalized spherical}
 \Phi_{\lambda,\textbf{m}}^{\nu}(r)=\frac{n!}{d_{\textbf{m}}}\det(\phi^\nu_{\lambda,(m_{i}-i+j)}(r))_{1\leq i,j\leq n}
\end{equation} 
\end{proposition}

\begin{proof}
The proof of the identity (\ref{K-compo}) is obvious by Schur's Lemma, since the Poisson transform is a $K$-equivariant map and $V_{\textbf{m}}$ is $K$-multiplicity free.\\  
Moreover we have
\begin{equation}\label{generalized spherical integral}
\Phi_{\lambda,\textbf{m}}^{\nu}(r)=\int_{U(n)}P_{\lambda,\nu}(rI,U)\phi_{\textbf{m}}(U){\rm d}U.
\end{equation}
We compute the above Hua-type integral in a manner similar to  the case \(\nu=0\) \cite{B3}. We give an outline of the proof. \\ Since the integrand in the right side of (\ref{generalized spherical integral}) is invariant  under the diagonal subgroup $L$  of $K$, we can use the Weyl integral formula to get 
\begin{equation*}
\begin{split}
\Phi_{\lambda,\textbf{m}}^{\nu}(r)=\frac{(1-r^2)^{\frac{n(i\lambda+n-\nu)}{2}}}{d_{\textbf{m}}}&\int_{\Gamma}\mid \det(I-re^{i\theta})\mid^{-(i\lambda+n-\nu)}(\det(I-re^{i\theta}))^{-\nu}\times\\
&\frac{A(\textbf{m}+\rho)(e^{i\theta})}
{A(\rho)(e^{i\theta})}
\mid A(\rho)(e^{i\theta})\mid^{2}{\rm d}\theta_{1}{\rm d}\theta_{2}\cdots {\rm d}\theta_{n},
\end{split}
\end{equation*}
where $e^{i\Theta}=diag(e^{i\theta_{1}},e^{i\theta_{2}},\cdots,e^{i\theta_{n}})$ and 
$\Gamma=[0,2\pi[\times\cdots\times [0,2\pi[$.\\
Next using the definition of the determinant we obtain (see \cite{B3})
\begin{equation}\label{generalized spherical 2}
\begin{split}
&\Phi_{\lambda,\textbf{m}}^{\nu}(r)=\frac{(1-r^2)^{\frac{n(i\lambda+n-\nu)}{2}}}{d_{\textbf{m}}}\sum\limits_{\sigma\in \mathcal{S}_{n}}\sum\limits_{\tau\in \mathcal{S}_{n}}sg(\tau)\times\\
\int_{\Gamma}\prod\limits_{j=1}^{n}\mid 1-re^{i\theta_{j}} & \mid^{-(i\lambda+n-\nu)}(1-re^{-i\theta_{j}})^{-\nu}e^{i<(\textbf{m}+\rho)-\tau.\rho,\sigma^{-1}.\theta> } {\rm d}\theta_{1}\cdots {\rm d}\theta_{n}.
\end{split}
\end{equation}
where $\mathcal{S}_{n}$ is the symmetric group of $n$ symbols, $sg$ the signature of a permutation and $<\mu,\theta>=\sum\limits_{j=1}^{n}\mu_{j}\theta_{j}$.\\
Thus, we are reduced  to compute  integrals  of the following type
\begin{equation*}
\int_{0}^{2\pi} (1-re^{i\theta})^{\frac{-(i\lambda+n+\nu)}{2}}(1-re^{i\theta})^{\frac{-(i\lambda+n-\nu)}{2}}e^{ik\theta}{\rm d}\theta,\quad k\in{\Z}. 
\end{equation*}
Using the binomial formula, a straightforward computation shows that the above integral equals
\begin{equation*}
r^{|k|}\frac{(\frac{i\lambda+n+\epsilon(k)\nu}{2})_{\mid k\mid}}{(1)_{\mid k\mid}}{}_2F_1(\frac{i\lambda+n-\epsilon(k)\nu}{2},\frac{i\lambda+n+\epsilon(k)\nu}{2}+\mid k\mid;1+\mid k\mid;r^{2}),
\end{equation*}
where $\epsilon(k)=1$ if $k\geq 0$ and $\epsilon(k)=-1$ if $k<0$.\\ 
Next, after the integration, every exponent in (\ref{generalized spherical 2}) gives the product
\begin{equation*}
\prod\limits_{j=1}^{n}(1-r^{2})^{\frac{-(i\lambda+n-\nu)}{2}}\phi_{\lambda,(m_{j}+\rho_{j}-\tau.\rho_{j})}^{\nu}(r),
\end{equation*}
independently on $\sigma$.\\
Thus
\begin{equation*}
\Phi_{\lambda,\textbf{m}}^{\nu}(r)=\frac{n!}{d_{\textbf{m}}}\sum\limits_{\tau\in \mathcal{S}_{n}}sg(\tau)\prod\limits_{j=1}^{n}\phi_{\lambda,(m_{j}+\rho_{j}-\tau.\rho_{j})}^{\nu}(r).
\end{equation*} 
The remaining part of the proof is the same as \(\nu=0\), so we omit it.
\end{proof}

As a direct consequence of Proposition \ref{prop gene spherical} and the main theorem in \cite{B} we get explicitly the series expansion of the eigenfunctions of the generalized Hua operator.
\begin{corollary}\label{serie expansion}
Let  $\lambda\in \C$ satisfying \(i\lambda \notin 2\Z^- +n-2\pm \nu\). If $F\in \mathcal{E}_{\lambda,\nu}({\cal D})$, then there exists a sequence of spherical harmonics functions $(f_{\textbf{m}})_{\textbf{m}\in \Lambda}$ such that  
\[
F(rU)=\sum_{\textbf{m}\in\Lambda} \Phi_{\lambda,\textbf{m}}^\nu(r)f_{\textbf{m}}(U),
\]
in $C^\infty([0,1[\times S)$. 
\end{corollary}

\section{The asymptotic behavior for the generalized spherical functions}
In this section we prove  a uniform asymptotic behavior for the generalized spherical functions \(\Phi_{\lambda,\textbf{m}}^\nu\).
\begin{k lemma}\label{key lemm}
Let $\nu \in \Z$ and $\lambda \in {\C}$ such that $i\lambda \notin 2\Z^- +n-2\pm \nu$ and $\Re(i\lambda)> n-1.$ Then 
\[
\Phi_{\lambda,\textbf{m}}^\nu(r)\sim \frac{\Gamma_\Omega(n)\Gamma_\Omega(i\lambda)}{\Gamma_\Omega(\frac{i\lambda+n+\nu}{2})\Gamma_\Omega(\frac{i\lambda+n-\nu}{2})} (1-r^2)^{\frac{n(n-\nu-i\lambda)}{2}},
\]
as $r$ goes to $1^-$ uniformly in $\textbf{m}\in\Lambda$.
\end{k lemma}
 
\begin{proof}
Recall that 
\(
\Phi_{\lambda,\textbf{m}}^{\nu}(r)=\frac{n!}{d_{\textbf{m}}}\det(\phi_{\lambda,(m_{i}-i+j)}^{\nu}(r))_{i,j},
\)
with 
\begin{eqnarray*}
&& \phi_{\lambda,(m_{i}-i+j)}^\nu(r)=r^{|m_{ij}|}(1-r^{2})^{\frac{i\lambda+n-\nu}{2}} \frac{(\frac{i\lambda+n+\epsilon_{ij}\nu}{2})_{|m_{ij}|}}{(1)_{|m_{ij}|}}\times\\
&& {}_2F_1(\frac{i\lambda+n-\epsilon_{ij}\nu}{2},\frac{i\lambda+n+\epsilon_{ij}\nu}{2}+|m_{ij}|;1+|m_{ij}|;r^{2}),
\end{eqnarray*}
where we have set \(m_{ij}=m_{i}-i+j\) and $\epsilon_{ij}=\epsilon(m_{ij})\).\\
Using the following identity on hypergeometric functions 
\begin{eqnarray*}
{}_2F_1(a,b;c;x)&=& \frac{\Gamma(c)\Gamma(c-a-b)}{\Gamma(c-a)\Gamma(c-b)}\;{}_2F_1(a,b;a+b-c+1;1-x)+\frac{\Gamma(c)\Gamma(a+b-c)}{\Gamma(a)\Gamma(b)}\times\\
&&(1-x)^{c-a-b}\;{}_2F_1(c-a,c-b;c-a-b+1;1-x),
\end{eqnarray*}
we can easily see that   
\begin{eqnarray*}
\phi_{\lambda,(m_{i}-i+j)}^{\nu}(r)\sim r^{|m_{ij}|}(1-r^{2})^{\frac{2-i\lambda-n-\nu}{2}}\frac{\Gamma(i\lambda+n-1)}{\Gamma(\frac{i\lambda+n+\nu}{2})\Gamma(\frac{i\lambda+n-\nu}{2})}\times \\
{}_2F_1(\frac{2-i\lambda-n-\epsilon_{ij}\nu}{2},\frac{2-i\lambda-n+\epsilon_{ij}\nu}{2}+|m_{ij}|;2-i\lambda-n;1-r^{2}),
\end{eqnarray*}
as $r$ goes to $1^-$, since $\Re(i\lambda)>n-1$.\\
Thus 
\begin{eqnarray*}
&\Phi_{\lambda,\textbf{m}}^\nu(r)\sim \frac{n!}{d_{\textbf{m}}}(1-r^{2})^{\frac{n(2-i\lambda-n-\nu)}{2}}\left(\frac{\Gamma(i\lambda+n-1)}{\Gamma(\frac{i\lambda+n+\nu}{2})\Gamma(\frac{i\lambda+n-\nu}{2})}\right)^n \times \\
&\det\left( r^{|m_{ij}|}{}_2F_1(\frac{2-i\lambda-n-\epsilon_{ij}\nu}{2},\frac{2-i\lambda-n+\epsilon_{ij}\nu}{2}+|m_{ij}|;2-i\lambda-n;1-r^{2})\right)_{i,j}, \textit{as} \; r\rightarrow 1^{-}.
\end{eqnarray*}
Below we will show that the above determinant does not depend on the sign of the integers $m_{ij}$. For  $\sigma\in \mathcal{S}_{n}$ , let $J_1$ (respectively \(J_2\)) denote the set of all $j$ such that $(m_j-j+\sigma(j))\geq 0$ (respectively the set of all $j$ such that $(m_j-j+\sigma(j))<0$). We have
\begin{eqnarray*}
&\det\left( r^{|m_{ij}|}{}_2F_1(\frac{2-i\lambda-n-\epsilon_{ij}\nu}{2},\frac{2-i\lambda-n+\epsilon_{ij}\nu}{2}+|m_{ij}|;2-i\lambda-n;1-r^{2})\right)_{i,j}\\
&= \displaystyle\sum_{\sigma \in \mathcal{S}_n}sg(\sigma)\displaystyle\prod_{j=1}^n r^{|m_{j\sigma(j)}|}{}_2F_1(\frac{2-i\lambda-n-\epsilon_{j\sigma(j)}\nu}{2},\frac{2-i\lambda-n+\epsilon_{j\sigma(j)}\nu}{2}+|m_{j\sigma(j)}|;2-i\lambda-n;1-r^{2})\\
&=\displaystyle\sum_{\sigma \in \mathcal{S}_n}sg(\sigma)\displaystyle\prod_{J_1}r^{(m_{j\sigma(j)})}{}_2F_1(\frac{2-i\lambda-n-\nu}{2},\frac{2-i\lambda-n+\nu}{2}+m_{j\sigma(j)};2-i\lambda-n;1-r^{2})\times \\
&\displaystyle\prod_{J_2}r^{-(m_{j\sigma(j)})}{}_2F_1(\frac{2-i\lambda-n+\nu}{2},\frac{2-i\lambda-n-\nu}{2}-m_{j\sigma(j)};2-i\lambda-n;1-r^{2})
\end{eqnarray*} 
Next, use the well known identity 
\begin{eqnarray*}
{}_2F_1(a,b;c;x)=(1-x)^{(c-a-b)}{}_2F_1(c-a,c-b;c;x),
\end{eqnarray*} 
to rewrite the product over \(J_2\) as 
\[
\prod_{J_2}r^{m_{j\sigma(j)}}{}_2F_1(\frac{2-i\lambda-n-\nu}{2},\frac{2-i\lambda-n+\nu}{2}+m_{j\sigma(j)};2-i\lambda-n;1-r^{2}),
\]
from which we get  
\begin{eqnarray*}
&\det\left( r^{|m_{ij}|}{}_2F_1(\frac{2-i\lambda-n-\epsilon_{ij}\nu}{2},\frac{2-i\lambda-n+\epsilon_{ij}\nu}{2}+|m_{ij}|;2-i\lambda-n;1-r^{2})\right)_{i,j}\\
& = r^{|\textbf{m}|}\det\left({}_2F_1(\frac{2-i\lambda-n-\nu}{2},\frac{2-i\lambda-n+\nu}{2}+m_{ij};2-i\lambda-n;1-r^{2})\right)_{i,j},
\end{eqnarray*}
with $|\textbf{m}|=m_1+m_2+\ldots+m_n$. Thus 
\begin{eqnarray*}\label{first estimate}
\Phi_{\lambda,\textbf{m}}^\nu(r)\sim \frac{n!}{d_{\textbf{m}}}r^{|\textbf{m}|}(1-r^{2})^{\frac{n(2-i\lambda-n-\nu)}{2}}\left(\frac{\Gamma(i\lambda+n-1)}{\Gamma(\frac{i\lambda+n+\nu}{2})\Gamma(\frac{i\lambda+n-\nu}{2})}\right)^n \times \\
\det\left(_2F_1(\frac{2-i\lambda-n-\nu}{2},\frac{2-i\lambda-n+\nu}{2}+m_{ij};2-i\lambda-n;1-r^{2})\right)_{i,j}, 
\end{eqnarray*}
as $r$ goes to $1^-$, for every $\textbf{m}\in \Lambda$.\\
Now, because  $i\lambda \notin 2\Z^- +n-2\pm \nu$ and $\Re(i\lambda)>n-1$, we may  apply Lemma A and Lemma B with $\alpha=\frac{2-i\lambda-n-\nu}{2}$, $\beta=\frac{2-i\lambda-n+\nu}{2}$ and $p_i=m_i-i$, to see that 
\begin{eqnarray*}\label{second estimate}
&\frac{1}{d_{\textbf{m}}}\det\left(_2F_1(\frac{2-i\lambda-n-\nu}{2},\frac{2-i\lambda-n+\nu}{2}+m_{ij};2-i\lambda-n;1-r^{2})\right)_{i,j}&\\ 
&\sim \gamma(\lambda,\nu)(1-r^2)^{n(n-1)},
\end{eqnarray*}
as $r$ goes to $1^-$, where the constant \(\gamma(\lambda,\nu)\) is given by  
\begin{eqnarray*}
\gamma(\lambda,\nu)&=&\prod_{k=1}^{n-1}(n-k)! \prod_{k=1}^{n-1} \frac{(\frac{2-i\lambda-n-\nu}{2}+k-1)^{n-k}(\frac{2-i\lambda-n+\nu}{2}+k-1)^{n-k}}{(-i\lambda-n+k+1)^{n-k}\displaystyle \prod_{j=1}^{n-k}(-i\lambda+k-j)_2}.
\end{eqnarray*}
Thus, we have
\begin{equation*}
\Phi_{\lambda,\textbf{m}}^\nu(r)\sim \left(\frac{\Gamma(i\lambda+n-1)}{\Gamma(\frac{i\lambda+n+\nu}{2})\Gamma(\frac{i\lambda+n-\nu}{2})}\right)^n \gamma(\lambda,\nu)r^{|\textbf{m}|}(1-r^{2})^{\frac{n(n-i\lambda-\nu)}{2}},
\end{equation*}
as $r$ goes to $1^-$, for every $\textbf{m}\in \Lambda$. \\
To finish the proof it suffices to prove  the following identity 
\begin{equation}\label{E9}
\left(\frac{\Gamma(i\lambda+n-1)}{\Gamma(\frac{i\lambda+n+\nu}{2})\Gamma(\frac{i\lambda+n-\nu}{2})}\right)^n \gamma(\lambda,\nu)=\frac{\Gamma_\Omega(n)\Gamma_\Omega(i\lambda)}{\Gamma_\Omega(\frac{i\lambda+n+\nu}{2})\Gamma_\Omega(\frac{i\lambda+n-\nu}{2})}.
\end{equation}
 A straightforward computation shows that  
\begin{equation*}
\begin{split}
&\left(\frac{\Gamma(i\lambda+n-1)}{\Gamma(\frac{i\lambda+n+\nu}{2})\Gamma(\frac{i\lambda+n-\nu}{2})}\right)^n \gamma(\lambda,\nu)=\frac{\Gamma_{\Omega}(n)\Gamma(i\lambda)}{\Gamma(\frac{i\lambda+n+\nu}{2})\Gamma(\frac{i\lambda+n-\nu}{2})} (i\lambda)_{n-1} \times \\ 
&\prod_{k=1}^{n-1}\frac{\Gamma(i\lambda+n-1)(\frac{i\lambda+n+\nu}{2}-k)^{n-k}(\frac{i\lambda+n-\nu}{2}-k)^{n-k}}{\Gamma(\frac{i\lambda+n+\nu}{2})\Gamma(\frac{i\lambda+n-\nu}{2})(i\lambda+n-1-k)^{n-k}(i\lambda-k+1)_{n-k}(i\lambda-k)_{n-k}},
\end{split}
\end{equation*}
and by using the identity $\displaystyle \prod_{k=1}^{n-1}(a-k)^{n-k}=\prod_{k=1}^{n-1}(a-k)_k$,  we find that  
\begin{eqnarray*}
\left(\frac{\Gamma(i\lambda+n-1)}{\Gamma(\frac{i\lambda+n+\nu}{2})\Gamma(\frac{i\lambda+n-\nu}{2})}\right)^n \gamma(\lambda,\nu)=\frac{\Gamma_{\Omega}(n)\Gamma_{\Omega}(i\lambda)}{\Gamma_{\Omega}(\frac{i\lambda+n+\nu}{2})\Gamma_{\Omega}(\frac{i\lambda+n-\nu}{2})}(i\lambda)_{n-1}\prod_{k=1}^{n-1}\frac{(i\lambda-k)_{n-1}}{(i\lambda-k+1)_{n-k}(i\lambda-k)_{n-k}}.
\end{eqnarray*}
Next proceeding by induction we easily obtain  
\[
(i\lambda)_{n-1}\prod_{k=1}^{n-1}\frac{(i\lambda-k)_{n-1}}{(i\lambda-k+1)_{n-k}(i\lambda-k)_{n-k}}=1, \;\; \forall n\geq 2,
\]
from which we get (\ref{E9}), as to be shown.
\end{proof}

\section{Proof of the mains results }
In this section we give the proof of our main results. We first establish the following estimate:
\begin{proposition}\label{prop CN}
Let $\nu \in \Z$ and $\lambda \in \C$ such that $i\lambda \notin 2\Z^- +n-2\pm \nu$ and $\Re(i\lambda)>n-1$.
\begin{itemize}
\item[(i)] If $\mu$ is a complex Borel measure on $S$ then there exists a positive constant $ \gamma(\lambda) $ such that
\begin{align*}
\|P_{\lambda,\nu} \mu\|_{*,1} \leq \gamma(\lambda)\|\mu\|.
\end{align*} 
\item[(ii)] If $f$ in $L^p(S)$, $1<p<\infty$, then  there exists a positive constant $ \gamma(\lambda) $ such that 
\begin{align*}
\|P_{\lambda,\nu} f\|_{*,p} \leq \gamma(\lambda)\|f\|_p.
\end{align*}
\end{itemize} 
\end{proposition} 

\begin{proof}
Let \( \mu\in \mathcal{M}(S)\). By using Fubini's theorem we find that
\begin{equation}
\int_S |(P_{\lambda,\nu}\mu)(rU)|{\rm d}U \leq \mid \mu(S)\mid \parallel P_{\lambda,\nu}\parallel_1
\end{equation}
Since $\Re(i\lambda)>n-1$ it then follows from the Forelli-Rudin generalized inequality (cf \cite{FK}) that 
\begin{equation}\label{E5}
 \parallel P_{\lambda,\nu}\parallel_1\leq \gamma(\lambda)(1-r^2)^\frac{n(n-\nu-\Re(i\lambda))}{2},
\end{equation}
for some positive constant \(\gamma(\lambda)\). This last inequality implies that \( \|P_{\lambda,\nu} \mu\|_{*,1} \leq \gamma(\lambda)\|\mu\|\). This proves (i).\\
To deduce (ii), one sees that (i) (with \(\nu=0\)) implies the other cases, because of (\ref{E5}), and since by the H\"{o}lder inequality we have 
\begin{eqnarray*}
|(P_{\lambda,\nu}f)|^p \leq (1-r^2)^{\frac{-n\nu}{2}} \parallel P_{\lambda,\nu}\parallel_1^{p-1} (P_{-i\Re(i\lambda),0}|f|^p).   
\end{eqnarray*}
This finishes the proof of Proposition \ref{prop CN}.
\end{proof}

\subsection{Proof of Theorem \ref{p=2}}
(i) The necessary condition follows from  Proposition \ref{prop CN}, for \(p=2\).\\
To prove the sufficiency condition, let $F\in \mathcal{E}_{\lambda,\nu}({\cal D})$ such that $\|F\|_{*,2}<\infty$. Since $i\lambda \notin 2\Z^- +n-2\pm \nu$ we know that  there exists a hyperfunction $f$ on the Shilov boundary $S$ such that $F=P_{\lambda,\nu} f$, by \cite{KZ}.\\
Next expanding $f$ into its $K$-type series $f=\sum_{\textbf{m}\in \Lambda}f_{\textbf{m}}$ and using Corollary \ref{serie expansion} we obtain
\[
F(rU)=\sum_{\textbf{m}\in\Lambda} \Phi_{\lambda,\textbf{m}}^\nu(r)f_{\textbf{m}}(U)
\]
in $ C^\infty([0,1[\times S)$.\\
From $\|F\|_{*,2}<\infty$ it follows that  
\[
(1-r^2)^{-n(n-\nu-\Re(i\lambda))}\sum_{\textbf{m}\in \Lambda}|\Phi_{\lambda,\textbf{m}}^\nu(r)|^2 \|f_{\textbf{m}}\|_2^2 <\infty,
\]
for every $r\in[0,1[$.\\
Using   the uniform asymptotic behavior of $\Phi_{\lambda,\textbf{m}}^\nu$ given by the Key Lemma, it follows that
\[
|c_\nu(\lambda)|\sum_{\textbf{m}\in \Lambda}\|f_{\textbf{m}}\|_2\leq \|P_{\lambda,\nu} f\|_{*,2}.
\]
This shows that $f\in L^2(S)$ and that  
\(
|c_\nu(\lambda)|\|f\|_2\leq \|P_{\lambda,\nu} f\|_{*,2}
\).\\
(ii) Let $F\in \mathcal{E}_\lambda^{2,\nu}({\cal D})$. By the first part of Theorem \ref{p=2} there exists $f \in L^2(S)$ such that $F=P_{\lambda,\nu} f$. Let $f=\sum_{\textbf{m}\in \Lambda}f_{\textbf{m}}$ be  its $K$-type series. It then follows from 
(\ref{K-compo}) that 
\[
F(rV)=\sum_{\textbf{m}\in\Lambda}\Phi_{\lambda,\textbf{m}}^\nu(r)f_{\textbf{m}}(V), 
\]
in $C^\infty([0,1[\times S)$.\\
For each $r\in [0,1[$, put
\[
g_r(U)=| c_\nu(\lambda)|^{-2}(1-r^2)^{-n(n-\nu-\Re(i\lambda))}\int_S \overline{P_{\lambda,\nu}(rV,U)}F(rV){\rm d}V.
\]
Replacing \(F\) by its series expansion and using (\ref{K-compo}), we easily find that
\[
g_r(U)=| c_\nu(\lambda)|^{-2}(1-r^2)^{-n(n-\nu-\Re(i\lambda))}\sum_{\textbf{m}\in \Lambda} |\Phi_{\lambda,\textbf{m}}^\nu(r)|^2f_{\textbf{m}}(U).
\]
Thus
\[\parallel g_r-f\parallel_2^2=\sum_{\textbf{m}\in \Lambda} \left\vert |c_\nu(\lambda)|^{-2}(1-r^2)^{-n(n-\nu-\Re(i\lambda))}|\Phi_{\lambda,\textbf{m}}^\nu(r)|^2-1\right\vert^2 \|f_{\textbf{m}}\|_2^2
,\]
this together with the Key Lemma imply \(\displaystyle\lim_{r\rightarrow 1^-}\parallel g_r-f\parallel_2=0\) and the proof of Theorem \ref{p=2} is finished.

\subsection{Proof of Theorem \ref{1<p<oo}}
The necessary conditions for $1\leq p <\infty$ follow from Proposition \ref{prop CN}.\\
To prove the sufficiency conditions, we consider \(\chi_j\) an approximate of the identity in \(\mathcal{C}(K)\), i.e. a sequence of nonnegative continuous functions on \(K\) having integral \(1\) and \(\lim_{j\rightarrow +\infty}\int_{K\setminus V}\chi_j(k){\rm d}k=0\), for every neighborhood \(V\) of the identity element in \(K\). For \(F\in \mathcal{E}_\lambda^{p,\nu}({\cal D})\) we set \(F_j(Z)=\int_{K}\chi_j(k)F(k^{-1}.Z){\rm d}k\). Then the sequence \((F_j)_j\) converges pointwise to \(F\) and  \(F_j\in \mathcal{E}_{\lambda,\nu}({\cal D})\) by the \(K\)-invariance of \(\widetilde{\mathcal{H}}_\nu\). For \(0\leq r<1\), we write  \(F^r(U)=F(rU)\). Then 
\begin{equation}\label{E6}
\|F^r_j\|_2\leq \|\chi_j\|_2.\|F^r\|_p,
\end{equation}
and 
\begin{equation}\label{E7}
\|F^r_j\|_p\leq \|F^r\|_p
\end{equation}
From (\ref{E6}) and Theorem \ref{p=2} it follows  that  for each $j$ there exists a function $f_j\in L^2(S)$ such that $ F_j= P_{\lambda,\nu} f_j$. Moreover
\[
f_j(U)=| c_\nu(\lambda)|^{-2}\lim_{r\rightarrow 1^-}(1-r^2)^{-n(n-\nu-\Re(i\lambda))}\int_S \overline{P_{\lambda,\nu}(rV,U)}F_j(rV){\rm d}V,
\]
in $L^{2}(S)$.\\
Following the same method as  in the proof of the trivial line bundle case (see \cite{B2}, \cite{B3}) we can show that $f_j$ lies in $L^p(S)$ with
\begin{equation*}
\| f_j\|_p \leq \gamma(\lambda)\mid c_\nu(\lambda)\mid^{-2} \| F\|_{\ast,p}.
\end{equation*} 
For $\phi\in L^{q}(S)$, $\frac{1}{p}+\frac{1}{q}=1$, let 
\[ 
T_j(\phi)=\int_{S}f_j(U)\overline{\phi(U)}{\rm d}U.
\]
Using (\ref{E7})  it follows that the  linear operators $T_j$ are uniformly bounded above by $ \gamma(\lambda)| c_\nu(\lambda)|^{-2}\| F\|_{*,p}$. By the Banach-Alaoglu-Bourbaki theorem, there exists a subsequence of bounded operators denote it by $(T_{j_k})_k$ which converges weak$^\ast$ to a bounded linear operator $T$ on $L^{q}(S)$ with 
\[
\|T\|\le \gamma(\lambda)| c_\nu(\lambda)|^{-2}\| F\|_{*,p},
\]
where $\|.\|$ stands for the operator norm.\\ 
Since $T$ is a bounded linear operator on $L^q(S)$ for $p>1$ and on $\mathcal{C}(S)$ for $p=1$, the Riesz representation theorem implies that for $p>1$ there exists a unique function $f\in L^{p}(S)$ such that 
$T(\phi)=\int_{S}f(U)\overline{\phi(U)}{\rm d}U$ with $\|f\|_{p}=\|T\|$ and for $p=1$  there exists a unique complex Borel measure $\mu$ on $S$ such that $T(\phi)=\int_{S}\phi(U){\rm d}\mu(U)$ with $\|T\|=\|\mu\|$. Thus
\begin{align*}
\| f\|_{p}\leq \gamma(\lambda)| c_\nu(\lambda)|^{-2}\| F\|_{*,p}\quad  \textit{for}\quad  p>1,
\end{align*}
and
\begin{align*}
\|\mu\|\leq \gamma(\lambda)| c_\nu(\lambda)|^{-2}\| F\|_{*,1}. 
\end{align*}
Fix $Z$ in $\mathcal{D}$ and set $\phi_{Z}(U)=P_{\lambda,\nu}(Z,U)$. Then $F_{j}(Z)=T_{j}(\phi_{Z})$. Since $F_{j}$ converges pointwise to $F$ and $T_j$ converges weak$^*$ to $T$, the result follows and the proof of Theorem \ref{1<p<oo} is finished.

\section{Appendix}
In this section we give the proof of Lemma A and Lemma B.
\subsection{Proof of Lemma A.}
\begin{proof}
Let \(A(r)\) be the \(n\times n\) matrix with entries \({}_2F_1(\alpha,\beta+p_i+j;\alpha+\beta;1-r^2)\).\\
We subtract each \(j\)th column from the \((j+1)\)th column and use  
the following well known identity on the hypergeometric functions:
\[
{}_2F_1(a,b;c;x)-{}_2F_1(a,b+1;c;x)=-\frac{a}{c}x{}_2F_1(a+1,b+1;c+1;x), 
\]
to see that \(\det(A(r))= (r^2-1)^{n-1}\left(\frac{\alpha}{\alpha+\beta}\right)^{n-1}\det(A^\prime(r))\), where the matrix 
\(A^\prime(r)\) is of the form 
\( A^\prime(r)=\left\{
\begin{array}{cl}
{}_2F_1(\alpha+1,\beta+p_i+j+1,\alpha+\beta+1; 1-r^2)&\text{if  \(1\leq i\leq n\) and \(1\leq j\leq n-1\)}\\
{}_2F_1(\alpha,\beta+p_i+n,\alpha+\beta; 1-r^2) &\text{for \(i=1,2,\cdots,n.\)}
\end{array}
\right.\)
Repeating this process \((n-1)\)-times we find that 
\begin{eqnarray*}
&&\det(A(r))=(r^2-1)^{\frac{n(n-1)}{2}}\prod_{k=1}^{n-1}\left(\frac{\alpha+k-1}{\alpha+\beta+k-1}\right)^{n-k}\\
&\times & \det(_2F_1(\alpha+n-j,\beta+p_i+n;\alpha+\beta+n-j;1-r^2))_{i,j},
\end{eqnarray*}
which proves Lemma A.
\end{proof}

\subsection{Proof of Lemma B} 
\begin{proof}
 The method used to prove Lemma 4.1 and Lemma 4.2 in \cite{B1}, can be generalized to our case, we shall give the main steps. Denoting by 
\[
\Psi(r)=\det(_2F_1(\alpha+n-j,\beta+p_i+n,\alpha+\beta+n-j,1-r^2))_{i,j}.
\]
and using the following  formula on  the hypergeometric functions 
\[
\left(\frac{d}{dx}\right)^q {}_2F_1(a,b;c;x)=\frac{(a)_m(b)_q}{(c)_q}{}_2F_1(a+q,b+q;c+q;x),
\]
we easily find that  
\begin{eqnarray*}
\left(\frac{d}{dr}\right)^q\Psi(r)_{\mid_{r=1}}=q!\sum_{\textbf{K}_q}\prod_{j=1}^n\frac{1}{k_j!}\frac{(\alpha+n-j)_{k_j}}{(\alpha+\beta+n-j)_{k_j}}\det((\beta+p_i+n)_{k_j})_{i,j},
\end{eqnarray*}
where $\textbf{K}_q$ is the set of all $n$-tuples $(k_1,k_2,\cdots,k_n)$ of non-negative integers such that $k_i\neq k_j$ and $k_1+k_2+\cdots+k_n=q$.\\
Set \(y_i=\beta+p_i+n\). Because  $\det((y_i)_{k_j})_{i,j}$ is an antisymmetric  polynomial in $(y_1,\cdots,y_n)$ of degree \(q\)  it is divisible by the Vandermonde determinant 
\(
\displaystyle\prod_{1\leq i<j \leq n}(y_i-y_j)=\prod_{1\leq i<j \leq n}(p_i-p_j)=d_\textbf{p} \prod_{l=1}^{n-1} l!
\).\\
Thus
\[
\left(\frac{d}{dr}\right)^q\Psi(r)_{\mid_{r=1}}=0, \;\; \forall q<\frac{n(n-1)}{2}.
\] 
Furthermore,
\begin{eqnarray*}
\left(\frac{d}{dr}\right)^{\frac{n(n-1)}{2}}\Psi(r)_{\mid_{r=1}}=d_\textbf{p} \prod_{l=1}^{n-1} l! \left(\frac{n(n-1)}{2}\right)! \sum_{\textbf{K}_{\frac{n(n-1)}{2}}}\prod_{j=1}^n\frac{1}{k_j!}\frac{(\alpha+n-j)_{k_j}}{(\alpha+\beta+n-j)_{k_j}}.
\end{eqnarray*}
Noting that each $(k_1,k_2,\cdots,k_n)\in \textbf{K}_{\frac{n(n-1)}{2}}$ is noting but a rearrangement of the set $\{0,1,\cdots,n-1\}$, we see that 
\begin{eqnarray*}
\left(\frac{d}{dr}\right)^{\frac{n(n-1)}{2}}\Psi(r)_{\mid_{r=1}}=d_\textbf{p} \left(\frac{n(n-1)}{2}\right)! \sum_{\sigma \in \mathcal{S}_n} sg(\sigma)\prod_{j=1}^n\frac{(\alpha+n-j)_{\sigma(j-1)}}{(\alpha+\beta+n-j)_{\sigma(j-1)}}.
\end{eqnarray*}
Therefore 
\begin{eqnarray}\label{E8}
\begin{split}
\frac{1}{d_\textbf{p}}\det(_2F_1(\alpha+n-j,&\beta+p_i+n;\alpha+\beta+n-j;1-r^2))_{i,j}\sim \left(\frac{n(n-1)}{2}\right)!\times \\\
& (1-r^2)^{{\frac{n(n-1)}{2}}}\det\left(\frac{(\alpha+n-j)_{i-1}}{(\alpha+\beta+n-j)_{i-1}}\right)_{1\leq i,j\leq n},
\end{split}
\end{eqnarray}
as $r$ goes to $1^-$.\\
Now, consider the left hand side of (\ref{E8}). We subtract each $(j+1)$th-column from the $j$th-column and use the identity 
\[
\frac{(\alpha+n-j)_{i-1}}{(\alpha+\beta+n-j)_{i-1}}-\frac{(\alpha+n-j-1)_{i-1}}{(\alpha+\beta+n-j-1)_{i-1}}=(i-1)\beta \frac{(\alpha+n-j)_{i-2}}{(\alpha+\beta+n-j-1)_i}.
\]
After this transformation the determinant is of the form
\begin{eqnarray*}
\det\left(\frac{(\alpha+n-j)_{i-1}}{(\alpha+\beta+n-j)_{i-1}}\right)_{1\leq i,j\leq n}=(-1)^{(n+1)} (n-1)! \prod_{j=1}^{n-1}\frac{\beta}{(\alpha+\beta+n-j-1)_2}\times \\
\det\left(\frac{(\alpha+n-j)_{i-1}}{(\alpha+\beta+n-j+1)_{i-1}}\right)_{1\leq i,j\leq (n-1)}.
\end{eqnarray*}
Repeating  this process, we get at the mth stage ( $1\leq m \leq n-2$)  the following formula:
\begin{equation*}
\begin{split}
&\det\left(\frac{(\alpha+n-j)_{i-1}}{(\alpha+\beta+n-j)_{i-1}}\right)_{1\leq i,j\leq n}=\prod_{k=1}^{m}(-1)^{n-k+2}(n-k)! \times \\
&\prod_{k=1}^{m}\prod_{j=1}^{n-k}\frac{\beta+k-1}{(\alpha+\beta+n+k-j-2)_2}
\det\left(\frac{(\alpha+n-j)_{i-1}}{(\alpha+\beta+n-j+m)_{i-1}}\right)_{1\leq i,j\leq (n-m)}.
\end{split}
\end{equation*}
Now taking $m=n-2$ we obtain the desired result.
\end{proof}

\end{document}